\documentclass[12pt]{amsart}

\usepackage{amssymb, amscd, txfonts}
\usepackage{graphicx}
\usepackage[all]{xy}

\numberwithin{equation}{section}

\newtheorem{theorem}{Theorem}[section]
\newtheorem{proposition}[theorem]{Proposition}
\newtheorem{lemma}[theorem]{Lemma}
\newtheorem{corollary}[theorem]{Corollary}

\theoremstyle{definition}

\newtheorem{example}[theorem]{Example}

\theoremstyle{remark}
\newtheorem{remark}[theorem]{Remark}
\newtheorem{claim}[theorem]{Claim}

\renewcommand{\ker}{\operatorname{Ker}}

\newcommand{\N}{\mathbb{N}}
\newcommand{\Z}{\mathbb{Z}}
\newcommand{\Q}{\mathbb{Q}}
\newcommand{\R}{\mathbb{R}}
\newcommand{\C}{\mathbb{C}}

\newcommand{\QZ}{\mathbb{Q}/\mathbb{Z}}

\newcommand{\emu}{{\mathbf e}_{\mu}}
\newcommand{\elambda}{{\mathbf e}_{\lambda}}

\newcommand{\SL}{{\rm SL}_2(\mathbb{Z})}
\newcommand{\Mp}{{\rm Mp}_2(\mathbb{Z})}
\newcommand{\pullL}{\uparrow_{L}^{L'}}

\newcommand{\pushLK}{\downarrow^{L}_{K}}
\newcommand{\ThetaK}{\Theta_{K^{+}}}
\newcommand{\ML}{M^{!}(L)}
\newcommand{\MLZ}{M^{!}(L)_{\mathbb{Z}}}
\newcommand{\MLQ}{M^{!}(L)_{\mathbb{Q}}}
\newcommand{\MLR}{M^{!}(L)_{R}}

\begin{document}

\title[]{Algebra of Borcherds products}
\author[]{Shouhei Ma}
\thanks{Supported by JSPS KAKENHI 17K14158, 20H00112, 21H00971.} 
\address{Department~of~Mathematics, Tokyo~Institute~of~Technology, Tokyo 152-8551, Japan}
\email{ma@math.titech.ac.jp}
\maketitle

\begin{abstract}
Borcherds lift for an even lattice of signature $(p, q)$ is 
a lifting from weakly holomorphic modular forms of 
weight $(p-q)/2$ for the Weil representation. 
We introduce a new product operation on the space of such modular forms and develop a basic theory.  
The product makes this space a finitely generated filtered associative algebra, 
without unit element and noncommutative in general.  
This is functorial with respect to embedding of lattices by the quasi-pullback. 
Moreover, the rational space of modular forms with rational principal part is closed under this product. 
In some examples with $p=2$, 
the multiplicative group of Borcherds products of integral weight 
forms a subring. 
\end{abstract}


\section{Introduction}

Since ancient, mathematicians have introduced and studied product structures 
on various mathematical objects. 
In this paper we define a product structure on a space of 
certain vector-valued modular forms of fixed weight attached to an integral quadratic form, 
that is functorial and that reflects some properties of the quadratic form. 

Let $L$ be an even lattice of signature $(p, q)$ with $p\leq q$ and  
$\rho_{L}$ be the Weil representation attached to the discriminant form of $L$. 
In \cite{Bo95}, \cite{Bo98}, Borcherds constructed a lifting from 
weakly holomorphic modular forms $f$ of weight $\sigma(L)/2=(p-q)/2$ and type $\rho_{L}$ 
to automorphic forms $\Phi(f)$ with remarkable singularity on the symmetric domain attached to $L$. 
When $p=2$ and the principal part of $f$ has integral coefficients, 
$\Phi(f)$ gives rise to a meromorphic modular form $\Psi(f)$ with infinite product expansion, 
known as \textit{Borcherds product}. 

The discovery of Borcherds has stimulated the study of 
weakly holomorphic modular forms of weight $\sigma(L)/2$ and type $\rho_{L}$.  
If we consider the space of such modular forms,  
say ${\ML}$, 
it is a priori just an infinite dimensional ${\C}$-linear space. 
The purpose of this paper is to introduce a product operation on the space ${\ML}$ 
and investigate its basic properties.  
This makes ${\ML}$ an associative ${\C}$-algebra, finitely generated and filtered but without unit element in general. 
Moreover, this product is functorial with respect to embedding of lattices by the so-called 
\textit{quasi-pullback} operation. 
This gives a link between quadratic forms and noncommutative rings. 

To state our result, 
we assume that $L$ has Witt index $p$ ($=$ maximal).  
Our construction requires the choice of a maximal isotropic sublattice $I$ of $L$. 
Then $K=I^{\perp}/I$ is an even negative-definite lattice of rank $-\sigma(L)$. 
Let ${\pushLK}$ be the pushforward operation from $\rho_{L}$ to $\rho_{K}$ (\S \ref{ssec: Weil representation}), 
and ${\ThetaK}(\tau)$ be the $\rho_{K^{+}}$-valued theta series of the positive-definite lattice $K^{+}=K(-1)$. 
In \S \ref{sec: product}, we define the $\Theta$-product of $f_{1}, f_{2} \in {\ML}$ with respect to $I$ by 
\begin{equation*}
f_{1} \ast f_{2} =  \langle f_{1}{\pushLK}, {\ThetaK} \rangle \cdot f_{2}.  
\end{equation*}
Then $f_{1} \ast f_{2}$ is again an element of ${\ML}$. 

In what follows, an \textit{associative algebra} is not assumed to have a unit element. 
Our basic results can be summarized as follows. 

\begin{theorem}\label{thm: main}
The $\Theta$-product $\ast$ makes ${\ML}$ a finitely generated filtered associative ${\C}$-algebra. 
The algebra ${\ML}$ has a unit element if and only if $L\simeq U\oplus \cdots \oplus U$. 
The algebra ${\ML}$ is commutative if and only if $L$ is unimodular. 
When $\sigma(L)<0$, the rational space ${\MLQ}\subset {\ML}$ of modular forms with rational principal part is closed under $\ast$. 

If $L'$ is a sublattice of $L$ of signature $(p, q')$ with $I_{{\Q}}\subset L'_{{\Q}}$, 
the map 
\begin{equation*}
{\ML} \to M^{!}(L'), \quad f\mapsto |I/I'|^{-1} \cdot f|_{L'}, 
\end{equation*}
is a homomorphism of ${\C}$-algebras, 
where $I'=I\cap L'$ and 
$f|_{L'}\in M^{!}(L')$ is the quasi-pullback of $f\in {\ML}$ 
as defined in \eqref{eqn: define quasi-pullback}. 
\end{theorem}

Here the filtration on ${\ML}$ is defined by the degree of principal part. 
$U$ stands for the integral hyperbolic plane, namely 
the even unimodular lattice of signature $(1, 1)$. 
The quasi-pullback map $|_{L'}\colon {\ML}\to M^{!}(L')$  
is an operation coming from quasi-pullback of Borcherds products (\cite{Bo95}, \cite{B-K-P-SB}, \cite{Ma}), 
which is a sort of renormalized restriction.  
The statements in Theorem \ref{thm: main} are proved in Propositions 
\ref{prop: associativity et al}, 
\ref{prop: unimodular commutative}, 
\ref{prop: RUE}, 
\ref{prop: rational part}, 
\ref{prop: f.g.}, and 
\ref{prop: functorial main}.

The algebra structure on ${\ML}$ requires the choice of $I$, 
but actually it depends only on the equivalence class of $I$ under a natural subgroup 
of the orthogonal group of $L$. 
Geometrically, when $p=2$, such equivalence classes correspond to 
maximal boundary components of the Baily-Borel compactification of the associated modular variety. 

In some special cases, $\Theta$-product is a quite simple operation. 
When $L$ is unimodular, so that $f_{1}, f_{2}$ and ${\ThetaK}=\theta_{K^{+}}$ are scalar-valued, 
$f_{1}\ast f_{2}$ is just the product $f_{1} \cdot \theta_{K^+} \cdot f_{2}$ (Example \ref{ex: unimodular}). 
When $I$ comes from $pU=U\oplus \cdots \oplus U$ embedded in $L$, 
so that we have a splitting $L=pU\oplus K$, 
$f_{1}, f_{2}$ correspond to weakly holomorphic Jacobi forms $\phi_{1}(\tau, Z), \phi_{2}(\tau, Z)$ 
of weight $0$ and index $K^+$ (see \cite{Gr}). 
Then the Jacobi form corresponding to $f_{1}\ast f_{2}$ is  
$\phi_{1}(\tau, 0)\cdot \phi_{2}(\tau, Z)$ (Example \ref{ex: Jacobi form}). 
In general, one can say that $\Theta$-product $\ast $ is a functorial extension of  
this simple product to all pairs $(L, I)$. 

Since the correspondence 
$(L, I)\mapsto ({\ML}, \ast_{I})$ 
is a functor, 
we expect that 
the complexity of the lattice $L$ (within fixed $p$) 
would be reflected in the complexity of the algebra ${\ML}$ in some way. 
The first examples are given in Theorem \ref{thm: main}: 
commutativity and existence of (two-sided) unit element. 
More widely, we show that 
${\ML}$ has a left unit element 
if it contains a certain modular form with very mild singularity 
(Proposition \ref{prop: LUE}). 
Some reflective modular forms provide typical examples of such a modular form 
(Examples \ref{ex: reflective 1} and \ref{ex: reflective 2}). 
This might remind us of 
Borcherds' philosophy \cite{Bo00b} that 
for $L$ Lorentzian, 
existence of a reflective modular form 
should be related to interesting property of the reflection group of $L$. 

In the same direction, we expect that 
the minimal number of generators of ${\ML}$ would reflect the size of $L$.  
We give lower and upper bounds on the number of generators, 
and deduce finiteness of lattices with bounded number of generators 
for fixed $(p, q)$ (Proposition \ref{prop: finiteness}). 
In the simple example $L=pU\oplus \langle -2 \rangle$, 
the algebra ${\ML}$ is generated by two basic reflective modular forms (Example \ref{ex: generator}).

The fact that the rational part ${\MLQ}$ is closed under $\ast$ enables us to define, when $p=2$, 
a "$\Theta$-product" of two Borcherds products as a third Borcherds product up to powers. 
For some $L$, even the group of Borcherds products of integral weight is closed under $\ast$, 
so it forms a subring. 

To conclude, the present article is a proposal of a new ring structure on ${\ML}$ 
and is devoted to the basic theory. 
Besides to find a concrete application, 
we would have at least four subjects to investigate in the theory: 
\begin{enumerate}
\item find further connection between the lattice $L$ and the algebra ${\ML}$. 
\item find an interesting ${\ML}$-module. 
\item whether the finiteness theorem holds even if $q$ is allowed to vary. 
\item find a geometric interpretation of the new "unusual" product of Borcherds products or its Lie bracket. 
\end{enumerate}

This paper is organized as follows. 
\S \ref{sec: preliminary} is recollection of modular forms for the Weil representation. 
In \S \ref{sec: product} we define $\Theta$-product. 
In \S \ref{sec: first property} we study first properties of the algebra ${\ML}$. 
In \S \ref{sec: f.g.} we prove finite generation. 
In \S \ref{sec: functorial} we prove functoriality. 
\S \ref{sec: first property} -- \S \ref{sec: functorial} may be read independently. 


\textit{Unless stated otherwise, every ring in this paper is not assumed to be commutative nor have a unit element.}

\section{Weil representation and modular forms}\label{sec: preliminary}

In this section we recall some basic facts about 
modular forms of Weil representation type 
following \cite{Bo98}, \cite{Br}.

\subsection{Weil representation}\label{ssec: Weil representation} 

Let $L$ be an even lattice, namely 
a free abelian group of finite rank equipped with a nondegenerate symmetric bilinear form 
$(\cdot , \cdot) \colon L\times L \to {\Z}$ such that 
$(l ,l)\in 2{\Z}$ for all $l\in L$. 
When $L$ has signature $(p, q)$, we write $\sigma(L)=p-q$. 
The dual lattice of $L$ is denoted by $L^{\vee}$. 
The quotient $A_L=L^{\vee}/L$ is called the \textit{discriminant group} of $L$, 
and is endowed with the canonical ${\QZ}$-valued quadratic form 
$q_L \colon A_L\to {\QZ}$, $q_L(x)=(x, x)/2+{\Z}$, 
called the \textit{discriminant form} of $L$. 
In general, a finite abelian group $A$ endowed with a nondegenerate quadratic form 
$q \colon A\to{\QZ}$ is called a \textit{finite quadratic module}. 
We will frequently abbreviate $(A, q)$ as $A$. 
Every finite quadratic module $A$ is isometric to the discriminant form of some even lattice $L$. 
We then write $\sigma(A)=[\sigma(L)]\in {\Z}/8$. 
We denote by ${\C}A$ the group ring of $A$. 
The standard basis vector of ${\C}A$ corresponding to an element $\lambda\in A$ 
will be denoted by ${\elambda}$. 

Let ${\Mp}$ be the metaplectic double cover of ${\SL}$.   
Elements of ${\Mp}$ are pairs $(M, \phi)$ where 
$M=\begin{pmatrix}a & b \\ c & d \end{pmatrix}\in {\SL}$ 
and $\phi$ is a holomorphic function on the upper half plane such that $\phi(\tau)^2=c\tau+d$. 
The group ${\Mp}$ is generated by 
$T = \left( \begin{pmatrix}1&1\\ 0&1\end{pmatrix}, 1 \right)$ and 
$S = \left( \begin{pmatrix}0&-1\\ 1&0\end{pmatrix}, \sqrt{\tau} \right)$,  
and the center of ${\Mp}$ is generated by 
$Z = S^{2} = \left( \begin{pmatrix}-1& 0\\ 0 & -1 \end{pmatrix}, \sqrt{-1} \right)$. 

The \textit{Weil representation} $\rho_A$ of ${\Mp}$ attached to a finite quadratic module $A$ 
is a unitary representation on ${\C}A$ defined by 
\begin{eqnarray*}
\rho_A(T)({\elambda}) & = & e(q(\lambda)){\elambda}, \\ 
\rho_A(S)({\elambda}) & = & 
\frac{e(-\sigma(A)/8)}{\sqrt{|A|}} \sum_{\mu\in A}e(-(\lambda, \mu)){\emu}.  
\end{eqnarray*}
Here $e(z)={\rm exp}(2\pi i z)$ for $z\in{\Q}/{\Z}$. 
We have 
\begin{equation*}
\rho_A(Z)({\elambda})  = e(-\sigma(A)/4)\mathbf{e}_{-\lambda}. 
\end{equation*}
We will also write $\rho_A=\rho_L$ when $A=A_{L}$ for an even lattice $L$.

Let $A(-1)$ be the $(-1)$-scaling of $A$, 
namely the same underlying abelian group with the quadratic form $q$ replaced by $-q$. 
Then $\rho_{A(-1)}$ is canonically isomorphic to the dual representation $\rho_A^{\vee}$ of $\rho_A$. 
The isomorphism is defined by sending the standard basis of ${\C}A(-1)$ 
to the dual basis $\{ \mathbf{e}_{\lambda}^{\vee} \}$ 
of the standard basis $\{ \mathbf{e}_{\lambda} \}$ of ${\C}A$ 
through the identification $A(-1)=A$ as abelian groups. 
We will tacitly identify $\rho_{A(-1)}=\rho_{A}^{\vee}$ in this way. 
 
Let $I\subset A$ be an isotropic subgroup. 
Then $A'=I^{\perp}/I$ inherits the structure of a finite quadratic module. 
Let $p:I^{\perp}\to A'$ be the projection. 
We define linear maps 
\begin{equation}\label{eqn: pull push}
\uparrow_{A'}^{A} : {\C}A' \to {\C}A, \qquad 
\downarrow_{A'}^{A} : {\C}A \to {\C}A', 
\end{equation}
called \textit{pullback} and \textit{pushforward} respectively, by 
\begin{equation*}
{\elambda}\uparrow_{A'}^{A}=\sum_{\mu\in p^{-1}(\lambda)}{\emu}, \qquad 
{\emu}\downarrow_{A'}^{A} = 
\begin{cases}
\mathbf{e}_{p(\mu)}, & \mu\in I^{\perp}, \\ 
0, & \mu\not\in I^{\perp},  
\end{cases}
\end{equation*}
for $\lambda\in A'$ and $\mu\in A$. 
Then 
$\uparrow_{A'}^{A}$ and $\downarrow_{A'}^{A}$ are homomorphisms between the Weil representations 
(see, e.g., \cite{Bo98}, \cite{Br}, \cite{Ma}). 
Note that 
$\downarrow^{A}_{A'}\circ \uparrow^{A}_{A'}$ 
is the scalar multiplication by $|I|$. 
Note also that 
$\uparrow_{A'}^{A}$ and $\downarrow_{A'}^{A}$ 
are adjoint to each other with respect to the standard Hermitian metrics on ${\C}A$ and ${\C}A'$. 
When $A=A_{L}$ for an even lattice $L$, 
the isotropic subgroup $I$ corresponds to 
the even overlattice $L\subset L'\subset L^{\vee}$ of $L$ with $L'/L=I$. 
Then $A'=A_{L'}$. 
In this situation, we will also write 
$\uparrow_{A'}^{A} = \uparrow_{L'}^{L}$ and 
$\downarrow_{A'}^{A}=\downarrow_{L'}^{L}$.

\subsection{Modular forms}\label{ssec: modular form}

Let $A$ be a finite quadratic module and 
let $k\in \frac{1}{2}{\Z}$ with $k\equiv \sigma(A)/2$ modulo $2{\Z}$. 
(We will be interested in the case $k\leq 0$.) 
A $ {\C}A$-valued holomorphic function 
$f$ on the upper half plane is called a \textit{weakly holomorphic modular form} 
of weight $k$ and type $\rho_A$ if it satisfies 
$f(M\tau)=\phi(\tau)^{2k}\rho_A(M, \phi)f(\tau)$ 
for every $(M, \phi)\in {\Mp}$ 
and is meromorphic at the cusp. 
We write 
\begin{equation*}
f(\tau) = \sum_{\lambda\in A} \sum_{n\in q(\lambda)+{\Z}} c_{\lambda}(n) q^{n}{\elambda} 
\end{equation*}
for the Fourier expansion of $f$ 
where $q^n=\exp (2\pi in\tau)$ for $n\in {\Q}$. 
By the invariance under $Z$, we have 
$c_{-\lambda}(n)=c_{\lambda}(n)$. 
The finite sum 
$\sum_{\lambda} \sum_{n<0} c_{\lambda}(n) q^{n}{\elambda}$ 
is called the \textit{principal part} of $f$. 
When $k<0$, $f$ is determined by its principal part; 
when $k=0$, $f$ is determined by its principal part and constant term. 
We write $M^{!}_{k}(\rho_A)$ for the space of weakly holomorphic modular forms of weight $k$ and type $\rho_{A}$. 
By the Borcherds duality theorem (\cite{Bo00a}, \cite{Bo00b}, \cite{Br}), 
which polynomial arises as the principal part of some $f\in M^!_k(\rho_A)$ 
is determined by certain cusp forms as follows. 
This will be used in \S \ref{sec: first property} and \S \ref{sec: f.g.}. 

\begin{theorem}[\cite{Bo00a}, \cite{Bo00b}, \cite{Br}]
Let 
$P=\sum_{\lambda, n}c_{\lambda}(n)q^n{\elambda}$ 
be a ${\C}A$-valued polynomial 
where $\lambda\in A$ and $n\in q(\lambda)+{\Z}$ with $n<0$, 
such that $c_{-\lambda}(n)=c_{\lambda}(n)$. 
Then $P$ is the principal part of a weakly holomorphic modular form 
of weight $k\equiv \sigma(A)/2$ mod $2{\Z}$ and type $\rho_{A}$ 
if and only if 
$\sum_{n<0}c_{\lambda}(n)a_{\lambda}(-n)=0$ 
for every cusp form 
$\sum_{\lambda, m} a_{\lambda}(m)q^{m}\mathbf{e}_{\lambda}^{\vee}$ 
of weight $2-k$ and type $\rho_{A}^{\vee}$. 
\end{theorem}

Theta series are typical examples of holomorphic modular forms of Weil representation type. 
Let $N$ be an even positive-definite lattice. 
By Borcherds \cite{Bo98}, 
the $\rho_N$-valued function 
\begin{equation*}
\Theta_{N}(\tau) 
= \sum_{l\in N^{\vee}} q^{(l, l)/2}\mathbf{e}_{[l]} 
= \sum_{\lambda, n}c_{\lambda}^{N}(n)q^{n}{\elambda},  
\end{equation*}
where $c_{\lambda}^{N}(n)$ is the number of vectors $l$ in $\lambda+N\subset N^{\vee}$ such that $(l, l)=2n$, 
is a holomorphic modular form of weight ${\rm rk}(N)/2$ and type $\rho_{N}$. 
All Fourier coefficients of $\Theta_{N}(\tau)$ are nonnegative integers. 
If $N'$ is an even overlattice of $N$, we have 
$\Theta_{N'}=\Theta_{N}\!\downarrow^{N}_{N'}$.

Let $L$ be an even lattice. 
For $A=A_{L}$ and $k=\sigma(L)/2$, 
we write 
\begin{equation*}
{\ML} = M^{!}_{\sigma(L)/2}(\rho_{L}). 
\end{equation*}
We especially write 
\begin{equation*}
M^{!}=M^{!}(U\oplus \cdots \oplus U)=M^{!}(\{ 0 \}), 
\end{equation*} 
which is just the space of scalar-valued weakly holomorphic modular forms of weight $0$. 
Then $M^{!}$ is the polynomial ring in the $j$-function 
$j(\tau)=q^{-1}+744+ \cdots$. 
It is a fundamental remark that for every even lattice $L$, ${\ML}$ is a $M^{!}$-module.

Let $p=2$. 
When the principal part of $f\in{\ML}$ has integral coefficients, 
Borcherds \cite{Bo95}, \cite{Bo98} constructed a meromorphic modular form 
$\Psi(f)$ on the Hermitian symmetric domain attached to $L$, 
now called a \textit{Borcherds product}, 
which has weight $c_{0}(0)/2\in {\Q}$ and whose divisor is 
a linear combination of Heegner divisors determined by the principal part of $f$. 
The lifting $f\mapsto \Psi(f)$ is multiplicative.


\section{$\Theta$-product}\label{sec: product} 

Let $L$ be an even lattice of signature $(p, q)$ with $p\leq q$ 
and assume that $L$ has Witt index $p$. 
We choose and fix a maximal ($=$ rank $p$, primitive) isotropic sublattice $I$ of $L$. 
In this section we define $\Theta$-product $\ast = \ast_{I}$ on the space 
${\ML}=M_{\sigma(L)/2}^{!}(\rho_{L})$ 
with respect to $I$, 
which makes ${\ML}$ an associative algebra.  
\S \ref{ssec: lattice lemma} is lattice-theoretic preliminary. 
$\Theta$-product is defined in \S \ref{ssec: theta product}. 
In \S \ref{ssec: example} we look at some examples.

\subsection{Preliminary}\label{ssec: lattice lemma}

We first prepare a lattice-theoretic lemma. 
We write $K=I^{\perp}\cap L/I$, which is an even negative-definite lattice of rank $-\sigma(L)$. 
We shall realize $K$ as an orthogonal direct summand of a canonical overlattice of $L$. 
Let $I^{\ast}=I_{{\Q}}\cap L^{\vee}$ be the primitive hull of $I$ in the dual lattice $L^{\vee}$. 
Then $L^{\ast}=\langle L, I^{\ast} \rangle$ is an even overlattice of $L$ with $L^{\ast}/L\simeq I^{\ast}/I$. 
For $rU=U\oplus \cdots \oplus U$ ($r$ times) 
we denote by $e_{1}, f_{1}, \cdots , e_{r}, f_{r}$ its standard basis, 
namely $(e_{i}, f_{j})=\delta_{ij}$ and $(e_i, e_j)=(f_i, f_j)=0$. 
We write $I_{r}=\langle e_{1}, \cdots, e_{r} \rangle$. 

\begin{lemma}\label{lem: overlattice split}
There exists an embedding $\varphi\colon pU \hookrightarrow L^{\ast}$ 
such that $\varphi(I_{p})=I^{\ast}$.  
In particular, we have 
$L^{\ast} = \varphi(pU)\oplus \varphi(pU)^{\perp} \simeq pU\oplus K$. 
The induced isometry 
$A_{L^{\ast}}\to A_{K}$ 
does not depend on the choice of $\varphi$.  
\end{lemma}

\begin{proof}
By the primitivity of $I^{\ast}$ in $L^{\vee}$, 
we have $(l, L^{\ast})=(l, L)={\Z}$ for any primitive vector $l$ in $I^{\ast}$. 
We take one such vector $l_1\in I^{\ast}$ and a vector $m_1\in L^{\ast}$ with $(l_1, m_1)=1$. 
Then $\langle l_1, m_1 \rangle \simeq U$ and 
we have a splitting $L^{\ast}=\langle l_1, m_1 \rangle \oplus L_{1}$ 
where $L_{1}=\langle l_1, m_1 \rangle^{\perp}\cap L^{\ast}$. 
The intersection 
$I_{1}=I^{\ast}\cap L_{1}$ satisfies $I^{\ast}=I_{1}\oplus {\Z}l_{1}$ 
and we have 
$(l, L_{1})=(l, L^{\ast})={\Z}$ for any primitive vector $l\in I_{1}$. 
Then we can repeat the same process for $I_{1}\subset L_{1}$. 
This eventually defines an embedding $\varphi\colon pU\hookrightarrow L^{\ast}$ with 
$\varphi(I_{p})=I^{\ast}$. 
We have natural isomorphisms 
\begin{equation*}
\varphi(pU)^{\perp}\cap L^{\ast} \stackrel{\simeq}{\to} 
(I^{\ast})^{\perp}\cap L^{\ast}/I^{\ast} = I^{\perp}\cap L/I =K. 
\end{equation*}

For the last assertion, we use the following construction. 
($I^{\ast}\subset L^{\ast}$ will be $I\subset L$ below.) 

\begin{claim}\label{claim: AL AK}
Let $L$ be an even lattice and $I\subset L$ be a primitive isotropic sublattice. 
Let $\varphi_{1}, \varphi_{2} \colon rU \hookrightarrow L$ 
be two embeddings such that  
$\varphi_{1}(I_{r})=\varphi_{2}(I_{r})=I$ 
and $\varphi_{1}|_{I_{r}}=\varphi_{2}|_{I_{r}}$. 
Then there exists 
an isometry $\gamma$ of $L$ which acts trivially on $I$, $K=I^{\perp}/I$ and $A_{L}$, 
such that $\varphi_{2}=\gamma \circ \varphi_{1}$. 
\end{claim}

The last assertion of Lemma \ref{lem: overlattice split} is deduced as follows. 
Let $\varphi_{1}, \varphi_{2} \colon rU \hookrightarrow L$ 
satisfy just $\varphi_{1}(I_{r})=\varphi_{2}(I_{r})=I$. 
We can find an isometry $\gamma'$ of $rU$ 
preserving $I_{r}$ and $\langle f_1, \cdots, f_r \rangle \simeq I_{r}^{\vee}$ 
such that 
$\varphi_{2}|_{I_{r}}=\varphi_{1}\circ \gamma'|_{I_{r}}$. 
Then there exists an isometry $\gamma$ of $L$ with properties as in Claim \ref{claim: AL AK} 
such that $\varphi_{2}=\gamma \circ \varphi_{1} \circ \gamma'$. 
If we write 
$K_{i}=\varphi_{i}(rU)^{\perp}\cap L$, 
then we have $\gamma(K_{1})=K_{2}$. 
The properties of $\gamma$ 
imply that the composition 
$A_{L}\to A_{K_{1}} \to A_{K}$ 
coincides with 
$A_{L}\to A_{K_{2}} \to A_{K}$,  
which is the desired assertion.  

We prove Claim \ref{claim: AL AK} by induction on $r$. 
When $r=1$, we let $l=\varphi_{i}(e_{1})$ and $m_{i}=\varphi_{i}(f_{1})$. 
Then as $\gamma$ we take the Eichler transvection $E_{l,m_{2}-m_{1}}$ (see, e.g., \cite{G-H-S}) 
which fixes $l$, sends $m_{1}$ to $m_{2}$, 
and acts trivially on $K$ and on $A_{L}$. 


For general $r$, 
let $rU=(r-1)U\oplus U$ be the apparent decomposition and let  
$\varphi_{i}'=\varphi_{i}|_{(r-1)U}$ and $I'=\varphi_{i}(I_{r-1})$. 
By induction, there exists an isometry $\gamma'$ of $L$ 
which acts trivially on $I'$, $(I')^{\perp}/I'$ and $A_{L}$, 
such that 
$\varphi_{2}'=\gamma' \circ \varphi_{1}'$. 
We show that $\gamma'$ also acts trivially on $I$ and $K=I^{\perp}/I$. 
Since $I/I'\subset (I')^{\perp}/I'$, $\gamma'$ preserves $I$. 
Since $K$ is a subquotient of $(I')^{\perp}/I'$, $\gamma$ acts trivially on $K$. 
We put $I''=\varphi_{i}({\Z}e_{r})$. 
Since 
\begin{equation*}
\gamma'(I'') = 
\gamma'(I\cap \varphi_{1}((r-1)U)^{\perp}) = 
I\cap \varphi_{2}((r-1)U)^{\perp} = I'', 
\end{equation*}
we find that $\gamma'$ preserves $I''$. 
Since the $\gamma'$-action on $I/I'$ is trivial, 
$\gamma'$ acts on $I''$ trivially. 
Thus $\gamma'$ acts on $I=I'\oplus I''$ trivially. 

We set 
$L'=\varphi_{2}((r-1)U)$ and 
$L''=(L')^{\perp}\cap L$.  
Then we can apply the result in the case $r=1$ to 
$\varphi_{1}''=\gamma' \circ \varphi_{1}|_{U}$, 
$\varphi_{2}''= \varphi_{2}|_{U}$, and 
$I''\subset L''$. 
This provides us with an isometry $\gamma_{L''}$ of $L''$ 
which acts trivially on $I''$, $(I'')^{\perp}/I''\simeq K$ and $A_{L''}\simeq A_{L}$, 
such that 
$\varphi_{2}''=\gamma_{L''}\circ \varphi_{1}''$. 
Now 
$\gamma=({\rm id}_{L'}\oplus \gamma_{L''})\circ \gamma'$ 
satisfies the desired properties. 
\end{proof}

\begin{remark}
The lattice $K$ can also be realized as a sublattice of $I^{\perp}\cap L$ as follows. 
We choose a basis $l_1, \cdots, l_p$ of $I$ and its dual basis 
$l_1^{\vee}, \cdots, l_p^{\vee}$ from $L^{\vee}$. 
We put $\tilde{K}=\langle l_1^{\vee}, \cdots, l_p^{\vee} \rangle^{\perp} \cap I^{\perp} \cap L$. 
By construction we have a splitting 
$I^{\perp}\cap L = I \oplus \tilde{K}$, 
so the projection gives an isometry $\tilde{K}\to K$. 
\end{remark}

\subsection{$\Theta$-product}\label{ssec: theta product}

We now define $\Theta$-product on ${\ML}$. 
We put $K^{+}=K(-1)$, which is an even positive-definite lattice of rank $-\sigma(L)$. 
We identify $A_{L^{\ast}}=A_{K}$ as in Lemma \ref{lem: overlattice split}. 
Let 
\begin{equation*}
{\pushLK} = \downarrow^{L}_{L^{\ast}} : A_{L} \to A_{L^{\ast}} = A_{K} 
\end{equation*}
be the pushforward operation defined in \eqref{eqn: pull push}. 
If $f\in {\ML}$, then 
$f{\pushLK}$ is an element of $M^{!}(K)$. 
We take the tensor product $f{\pushLK}\otimes {\ThetaK}$ with the theta series ${\ThetaK}$.  
This is a weakly holomorphic modular form of weight $0$ and type 
$\rho_{K}\otimes \rho_{K^{+}} \simeq \rho_{K}\otimes \rho_{K}^{\vee}$. 
Taking the contraction $\rho_{K}\otimes \rho_{K}^{\vee} \to {\C}$ 
produces a scalar-valued weakly holomorphic modular form of weight $0$, 
namely an element of $M^{!}$.  
We denote this modular function by 
\begin{equation*}
\xi(f) = \langle f{\pushLK}, {\ThetaK} \rangle \; \; \in M^{!}.  
\end{equation*}
The map $\xi \colon {\ML}\to M^{!}$ is $M^{!}$-linear. 

Now if $f_1, f_2\in {\ML}$, we define 
\begin{equation*}
f_{1} \ast f_{2} 
:= \; \xi(f_{1}) \cdot f_{2} = \langle f_{1}{\pushLK}, {\ThetaK} \rangle \cdot f_2. 
\end{equation*}
This is again an element of ${\ML}$. 
The map 
\begin{equation*}
\ast : {\ML} \times {\ML} \to {\ML} 
\end{equation*}
is $M^!$-bilinear. 
We write $\ast=\ast_{I}$ when we need to specify $I$. 

Explicitly, if 
$f_{i}(\tau)=\sum_{\lambda, n}c_{\lambda}^{i}(n)q^{n}{\elambda}$ 
for $i=1, 2$ and 
${\ThetaK}(\tau)=\sum_{\nu, m} c_{\nu}^{K}(m)q^{m}\mathbf{e}_{\nu}^{\vee}$, 
the Fourier coefficients of 
$f_{1}\ast f_{2}=\sum_{\lambda, n}c_{\lambda}(n)q^{n}{\elambda}$ 
are given by 
\begin{equation}\label{eqn: explicit coefficient}
c_{\lambda}(n) = 
\sum_{m+l+k=n}\sum_{\mu\in J^{\perp}} 
c^{1}_{\mu}(m) \cdot c^{K}_{p(\mu)}(l) \cdot c^{2}_{\lambda}(k). 
\end{equation}
Here 
$J=I^{\ast}/I\subset A_{L}$ 
and $p:J^{\perp}\to A_{K}$ is the projection. 
Note that even coefficients of $f_1, f_2$ in $n>0$, 
sometimes not being paid much attention, 
may contribute to the principal part of $f_1\ast f_2$.

\begin{proposition}\label{prop: associativity et al}
We have 
%
%
\begin{equation*}
(f_1\ast f_2)\ast f_3 = f_1\ast (f_2\ast f_3) 
\end{equation*}
for $f_1, f_2, f_3 \in {\ML}$. 
Therefore $\Theta$-product $\ast$ makes ${\ML}$ an associative ${\C}$-algebra. 
Moreover, the map $\xi \colon {\ML}\to M^{!}$ is a ring homomorphism.  
%
\end{proposition}

\begin{proof}
For the first assertion, we have 
\begin{eqnarray*}
(f_1\ast f_2)\ast f_3 
& =  & 
\xi(f_1\ast f_2) \cdot f_3  
\; = \; \xi(\xi(f_1) \cdot f_2) \cdot f_3 \\  
& = & 
\xi(f_1) \cdot \xi(f_2) \cdot f_3  
\; = \; \xi(f_1) \cdot (f_2 \ast f_3) \\
& = & 
f_{1}\ast (f_2 \ast f_3). 
\end{eqnarray*}
For the second assertion, we calculate  
\begin{equation*}
\xi(f_{1} \ast f_{2}) = 
\xi(\xi(f_{1})\cdot f_{2}) = 
\xi(f_{1})\cdot \xi(f_{2}).  
\end{equation*}
Thus $\xi$ preserves the products. 
\end{proof}

The algebra ${\ML}$ has the following filtration. 
For a natural number $d$ we denote by 
$M^{!}(L)_{d}\subset {\ML}$ 
the subspace of modular forms $f$ 
whose principal part has degree $\leq d$. 
Then we have 
\begin{equation*}
M^{!}(L)_{d} \ast M^{!}(L)_{d'} \subset M^{!}(L)_{d+d'}. 
\end{equation*}
Hence ${\ML}$ is a filtered algebra with this filtration. 

By the general theory of associative algebra, 
${\ML}$ has the structure of a Lie algebra by the commutator bracket 
\begin{equation}\label{eqn: Lie bracket}
[ f_1, f_2 ] = f_1\ast f_2 - f_2\ast f_1. 
\end{equation}
Since $\xi$ is a ring homomorphism and $M^!$ is commutative, 
these brackets are annihilated by $\xi$. 
We have  
\begin{equation*}
f_1\ast f_2\ast f_3 = f_2\ast f_1\ast f_3 
\end{equation*}
for $f_1, f_2, f_3 \in {\ML}$. 

The above construction requires the choice of 
a maximal isotropic sublattice $I$, 
so we should write $\ast=\ast_{I}$, $\xi=\xi_{I}$ and ${\ML}=M^{!}(L, I)$ 
when we want to specify this dependence. 
In fact, the freedom of choice is finite. 
If $\gamma\colon L\to L$ is an isometry of $L$, 
then $\gamma$ acts on $A_{L}$. 
Since the induced action on ${\C}A_{L}$ preserves the Weil representation $\rho_L$, 
$\gamma$ acts on ${\ML}$. 
We have 
$\xi_{\gamma I}(\gamma f)=\xi_{I}(f)$ 
and so 
\begin{equation*}
(\gamma f_{1})\ast_{\gamma I} (\gamma f_{2}) 
= \gamma (f_{1}\ast_{I} f_{2}). 
\end{equation*}
In other words, the action of $\gamma$ on ${\ML}$ gives an isomorphism 
\begin{equation*}
\gamma : M^!(L, I) \to M^!(L, \gamma I) 
\end{equation*}
of algebras. 
In particular, 
when $\gamma$ acts trivially on $A_{L}$, 
its action on ${\ML}$ is also trivial, 
so we have 
$M^{!}(L, I)=M^{!}(L, \gamma I)$ 
as algebras. 

To summarize, 
if ${\rm O}(L)$ is the orthogonal group of $L$ and 
$\Gamma_{L}<{\rm O}(L)$ is the kernel of the reduction map ${\rm O}(L)\to {\rm O}(A_{L})$, 
then $M^!(L, I)$ depends only on the $\Gamma_{L}$-equivalence class of $I$. 
Moreover, its isomorphism class depends only on the ${\rm O}(L)$-equivalence class of $I$. 
In particular, we have only finitely many algebra structures $M^!(L, I)$ on ${\ML}$ 
for a fixed lattice $L$. 

Geometrically, the $\Gamma_{L}$-equivalence class of $I$ corresponds more or less to 
a boundary component of some compactification of the locally symmetric space associated to $\Gamma_{L}$. 
(For example, when $p=2$, a boundary curve in the Baily-Borel compactification.) 
Perhaps this geometric picture might lead one to wonder whether it is possible to interpolate 
$M^!(L, I)$ and $M^!(L, I')$ for $I\not\sim I'$ 
by some continuous family of algebraic objects.

\subsection{Examples}\label{ssec: example}

We look at $\Theta$-product in some examples. 

\begin{example}\label{ex: unimodular}
Assume that $L$ is unimodular. 
Then $8 | \sigma(L)$. 
Modular forms of type $\rho_{L}$ are just scalar-valued modular forms. 
For any maximal isotropic sublattice $I$ we can find a splitting $L\simeq pU \oplus K$ with $I\subset pU$, 
and $K$ is also unimodular. 
In particular, ${\pushLK}$ is identity and 
${\ThetaK}=\theta_{K^{+}}$ is also scalar-valued. 
In this case, $\Theta$-product is just the product 
\begin{equation*}
f_1 \ast_{I} f_{2} = f_1 \cdot \theta_{K^{+}} \cdot f_2 
\end{equation*}
for $f_1, f_2\in {\ML}$. 
This shows that ${\ML}$ is commutative and has no zero divisor.  
Furthermore, ${\ML}$ has no unit element unless when $L=pU$. 
Indeed, if $f\in {\ML}$ is a unit element, then $f\cdot \theta_{K^{+}}=1$, 
but this is impossible when $K\ne \{ 0 \}$ 
because 
then $f$ would be a holomorphic modular form of negative weight. 
\end{example}

\begin{example}\label{ex: Jacobi form}
More generally, assume that we have a splitting 
$L=pU\oplus K$ with $I\subset pU$ ($K$ not necessarily unimodular). 
This is equivalent to $I=I^{\ast}$. 
In this situation, modular forms of type $\rho_{L}=\rho_{K}$ 
correspond to Jacobi forms of index $K^{+}$ as follows (see \cite{Gr} for more detail). 
Let 
$\Theta_{K^{+}}(\tau, Z) = \sum_{\lambda\in A_{K}} \theta_{K^{+}+\lambda}(\tau, Z)\mathbf{e}_{\lambda}^{\vee}$ 
be the $\rho_{K^{+}}$-valued Jacobi theta series. 
If $f(\tau)=\sum_{\lambda\in A_K}f_{\lambda}(\tau){\elambda}$ 
is a weakly holomorphic modular form of weight $\sigma(L)/2$ and type $\rho_{K}$, 
the function 
\begin{equation*}
\phi(\tau, Z) 
= \langle f(\tau), \, \Theta_{K^{+}}(\tau, Z) \rangle 
= \sum_{\lambda\in A_{K}} f_{\lambda}(\tau)\theta_{K^{+}+\lambda}(\tau, Z) 
\end{equation*}
given by the contraction 
$\rho_K\otimes \rho_K^{\vee} \to {\C}$ 
is a weakly holomorphic Jacobi form of weight $0$ and index $K^{+}$. 
This gives a one-to-one correspondence between two such forms. 
Note that 
the restriction $\phi(\tau, 0)$ of $\phi(\tau, Z)$ to $Z=0$  
is just the modular function $\xi(f)$ because 
$\Theta_{K^{+}}(\tau, 0) = \Theta_{K^{+}}(\tau)$. 

Now let $f_1, f_2 \in {\ML}$ 
and $\phi_1, \phi_2$ be the corresponding Jacobi forms. 
Then the Jacobi form corresponding to $f_1\ast_{I} f_2$ is 
\begin{equation*}
\phi_1(\tau, 0) \cdot \phi_2(\tau, Z). 
\end{equation*}
Indeed, we have 
\begin{eqnarray*}
\langle f_1\ast f_2(\tau), \: \Theta_{K^{+}}(\tau, Z) \rangle 
& = & 
\langle \xi(f_{1})(\tau) \cdot f_2(\tau), \: \Theta_{K^{+}}(\tau, Z) \rangle \\ 
& = &  
 \xi(f_{1})(\tau) \cdot \langle f_2(\tau), \: \Theta_{K^{+}}(\tau, Z) \rangle \\ 
& = &  
\phi_1(\tau, 0) \cdot \phi_2(\tau, Z). 
\end{eqnarray*}
Thus Jacobi form interpretation of $\Theta$-product is simple: 
substitute $Z=0$ into $\phi_1$ to obtain a scalar-valued modular function, 
and multiply it to $\phi_2$. 
$\Theta$-product for general $(L, I)$, not necessarily coming from $pU\hookrightarrow L$, 
can be thought of as a functorial extension of this simple operation using the pushforward operation $\downarrow^{L}_{K}$. 
\end{example}


\section{First properties}\label{sec: first property}

In this section we study some first properties of the algebra ${\ML}$. 
The reference maximal isotropic sublattice $I\subset L$ is fixed throughout and we write $\ast=\ast_{I}$. 
In \S \ref{ssec: annihilator} we study the left annihilator ideal of ${\ML}$, 
which plays a basic role in the study of ${\ML}$. 
In \S \ref{ssec: unit} we study existence/nonexistence of unit element. 
In \S \ref{ssec: integral part} we prove that the rational part ${\MLQ}$ is closed under $\ast$. 
\S \ref{ssec: unit} should be read after \S \ref{ssec: annihilator}, 
but \S \ref{ssec: integral part} may be read independently.

\subsection{Left annihilator}\label{ssec: annihilator} 

The left annihilator ideal of ${\ML}$ 
is a two-sided ideal of ${\ML}$. 
Since ${\ML}$ is torsion-free as a $M^{!}$-module, 
this coincides with the kernel of $\xi \colon {\ML}\to M^{!}$, 
which we denote by 
\begin{equation*}
\Theta^{\perp} = 
\{ \: f\in{\ML} \: | \: \langle f{\pushLK}, {\ThetaK} \rangle = 0 \: \}. 
\end{equation*}
This is also a sub $M^!$-module. 
Note that $\Theta^{\perp}$ also coincides with the left annihilator of any \textit{fixed} $g\ne0 \in {\ML}$. 
We have 
$(\Theta^{\perp})^{2}=0$. 
The ideal $\Theta^{\perp}$ is the maximal nilpotent ideal of ${\ML}$, 
consisting of all nilpotent elements of ${\ML}$.

\begin{proposition}\label{prop: theta kernel basic}
The quotient ring ${\ML}/\Theta^{\perp}$ is canonically identified with 
a nonzero ideal of the polynomial ring $M^{!}={\C}[j]$.  
Every homomorphism from ${\ML}$ to a ring without nonzero nilpotent element 
factors through ${\ML}\to {\ML}/\Theta^{\perp}$. 
\end{proposition}

\begin{proof}
By the definition $\Theta^{\perp}={\rm Ker}(\xi)$, 
the quotient ${\ML}/\Theta^{\perp}$ is identified with 
the image $\xi({\ML})\subset M^!$ of $\xi$. 
Since $\xi$ is a $M^!$-linear map, 
$\xi({\ML})$ is an ideal of $M^!$. 
We shall show that $\xi$ is a nonzero map. 
Since the map ${\pushLK}\colon {\ML}\to M^{!}(K)$ is surjective, 
it suffices to check that the map 
$\langle \cdot , {\ThetaK} \rangle \colon M^!(K)\to M^{!}$ is nonzero. 
This can be seen, e.g., by taking a modular form $f\in M^!(K)$ 
with Fourier expansion of the form 
$f(\tau)=q^{n}\mathbf{e}_{0}+o(q^{n})$ 
for some negative integer $n$,  
which is possible as guaranteed by Lemma \ref{lem: leading term}.  
The last assertion follows by a standard argument. 
\end{proof}

\begin{proposition}\label{prop: unimodular commutative} 
The following three conditions are equivalent. 

(1) $L$ is unimodular. 

(2) ${\ML}$ is commutative. 

(3) $\Theta^{\perp} = \{ 0 \}$. 

\noindent
Moreover, if $\Theta^{\perp}\ne \{ 0 \}$, we have $\dim \Theta^{\perp} = \infty$. 
\end{proposition}

\begin{proof}
(1) $\Rightarrow$ (2), (3) is observed in Example \ref{ex: unimodular}. 
(3) $\Rightarrow$ (2) holds because $[f_1, f_2]\in \Theta^{\perp}$. 
We check (2) $\Rightarrow$ (3). 
If $\Theta^{\perp}\ne \{ 0 \}$, 
we take $f_1\ne 0 \in \Theta^{\perp}$ and $f_2\not\in \Theta^{\perp}$. 
Then $f_1\ast f_2=0$ but $f_{2}\ast f_{1}\ne 0$, 
so ${\ML}$ is not commutative. 

Finally, we prove (3) $\Rightarrow$ (1). 
Suppose that $L$ is not unimodular. 
We shall show that $\dim \Theta^{\perp} = \infty$. 
We consider separately according to whether $K$ is unimodular or not. 
When $K$ is unimodular, 
$\Theta^{\perp}$ coincides with the kernel of the pushforward 
${\pushLK}\colon {\ML} \to M^{!}(K)$. 
We show that 
$\dim {\rm Ker}({\pushLK})= \infty$. 
The map ${\pushLK}$ preserves the degree filtration, 
namely 
$M^!(L)_{d}{\pushLK} \subset M^!(K)_d$. 
By the Borcherds duality theorem, we have 
\begin{eqnarray*}
\dim M^!(L)_d & = & |A_L/\pm 1| \cdot d+O(1), \\
\dim M^!(K)_d & = & 1\cdot d+O(1), 
\end{eqnarray*}
as $d$ grows. 
Therefore 
\begin{equation*}
\dim (\ker({\pushLK})\cap M^!(L)_d) \: \geq \: (|A_L/\pm1|-1)\cdot d + O(1) \to \infty 
\end{equation*}
as $d\to \infty$. 
Here $|A_L/\pm1|>1$ because $A_{L}\ne \{ 0 \}$. 

When $K$ is not unimodular, we can still argue similarly. 
The map 
${\pushLK}\colon {\ML}\to M^!(K)$ 
is surjective as 
the composition $\downarrow^{L}_{K} \circ \uparrow^{L}_{K}$ is a nonzero scalar multiplication.  
Therefore it is sufficient to show that the subspace 
$\ker \langle \cdot, {\ThetaK} \rangle$ of $M^!(K)$ 
has dimension $\infty$. 
The map 
$\langle \cdot, {\ThetaK} \rangle \colon M^!(K)\to M^!$ 
preserves the degree filtration, 
so we have similarly 
\begin{equation*}
\dim (\ker \langle \cdot, {\ThetaK} \rangle \cap M^!(K)_d) \: \geq \: (|A_K/\pm 1|-1)\cdot d + O(1) \to \infty 
\end{equation*}
as $d\to \infty$. 
This finishes the proof of (3) $\Rightarrow$ (1). 
\end{proof}

By Proposition \ref{prop: theta kernel basic}, 
${\ML}$ is decomposed into two parts: 
the ideal $\xi({\ML})$ in the polynomial ring $M^{!}={\C}[j]$, 
and the left annihilator $\Theta^{\perp}$. 
By the proof of (2) $\Rightarrow$ (3) in Proposition \ref{prop: unimodular commutative},  
the Lie brackets $[f, g]$ generate a large part of $\Theta^{\perp}$ 
containing at least $\xi({\ML})\cdot \Theta^{\perp}$. 
In \S \ref{sec: functorial} we will see that 
the kernels of the quasi-pullback maps provide 
natural examples of two-sided ideals contained in $\Theta^{\perp}$. 

\begin{remark}
We have only studied the left annihilator. 
The right annihilator of a fixed $f\in {\ML}$ 
coincides with the whole ${\ML}$ if $f\in \Theta^{\perp}$, 
while it is $\{ 0 \}$ if $f\not\in \Theta^{\perp}$. 
\end{remark}

\subsection{Unit element}\label{ssec: unit}

Next we study existence/nonexistence of unit element. 
Right unit element exists only in the apparent case. 

\begin{proposition}\label{prop: RUE}
${\ML}$ has a right unit element if and only if $L=pU$. 
In this case it is actually the two-sided unit element. 
\end{proposition}

\begin{proof}
It suffices to verify the ``only if'' direction. 
Let $g\in {\ML}$ be a right unit element. 
If $L$ is not unimodular, we can take $f\ne 0 \in \Theta^{\perp}$ by Proposition \ref{prop: unimodular commutative}. 
Then $f\ast g=0\ne f$, which is absurd. 
So $L$ must be unimodular. 
Then the assertion follows from the last part of Example \ref{ex: unimodular}. 
\end{proof}

On the other hand, 
left unit element, though still relatively rare, 
exists in more cases. 
They are exactly modular forms $f\in {\ML}$ with $\xi(f)=1$. 
In particular, if $f$ is a left unit element, 
every element of $f+\Theta^{\perp}$ is so.

\begin{proposition}\label{prop: LUE}
(1) ${\ML}$ has a left unit element if and only if the homomorphism 
$\xi\colon {\ML}\to M^{!}$ is surjective. 
This always holds when $\sigma(L)=0$.

(2) If there exists a modular form $f\in {\ML}\backslash \Theta^{\perp}$ with $f(\tau)=o(q^{-1})$, 
then ${\ML}$ has a left unit element. 
Such a modular form $f$ exists only when $|\sigma(L)|<24$. 
\end{proposition}

\begin{proof}
The first assertion of (1) holds because $\xi$ is $M^!$-linear. 
When $\sigma(L)=0$, 
we have $M^!(K)=M^!$ and 
$\xi={\pushLK}\colon {\ML}\to M^{!}(K)$ is surjective. 

Next we prove (2). 
If $f=o(q^{-1})$, we have 
$\xi(f) = o(q^{-1})$. 
Since Fourier expansion of elements of $M^!$ have only integral powers of $q$, 
we have in fact $\xi(f) = O(1)$. 
Hence $\xi(f)$ is a holomorphic modular function, 
namely a constant, 
which is nonzero by our assumption $f\not\in \Theta^{\perp}$. 
As for the last assertion of (2), 
we consider the product $\Delta f$ with the $\Delta$-function. 
This is a cusp form , so its weight $12+\sigma(L)/2$ must be positive. 
\end{proof}

The condition $f\not\in \Theta^{\perp}$ in Proposition \ref{prop: LUE} (2) is satisfied when 
the principal part of $f{\pushLK}$ has nonnegative (at least one nonzero) coefficients. 
Indeed, $\Theta_{K^{+}}(\tau)=\mathbf{e}_{0}^{\vee}+o(1)$ has nonnegative coefficients 
and the coefficient $c_{0}(0)$ of $f{\pushLK}$ is positive (\cite{Br}, \cite{B-K}), 
so $\xi(f)$ has nonzero constant term. 

Some reflective modular forms provide 
typical examples of modular forms as in Proposition \ref{prop: LUE} (2). 

\begin{example}\label{ex: reflective 1}
Let $L= pU \oplus \langle -2 \rangle$. 
Then $K^{+}= \langle 2 \rangle$. 
Let $\phi_{0,1}$ be the weak Jacobi form of weight $0$ and index $1$ 
constructed by Eichler-Zagier in \cite{E-Z} Theorem 9.3. 
The corresponding modular form in ${\ML}$ has Fourier expansion 
$f(\tau)=q^{-1/4}\mathbf{e}_{1}+10\mathbf{e}_{0}+o(1)$  
where $\mathbf{e}_{i}$ is the basis vector of ${\C}A_{L}$ corresponding to $[i]\in {\Z}/2\simeq A_{L}$. 
This modular form satisfies the condition in Proposition \ref{prop: LUE} (2). 
We will return to this example in Example \ref{ex: generator}. 
\end{example}

\begin{example}\label{ex: reflective 2}
More generally, let $L=pU\oplus \langle -2t \rangle$. 
Then $K=K_{t}=\langle -2t \rangle$. 
Eichler-Zagier's Jacobi form $\phi_{0,1}$ was generalized by Gritsenko-Nikulin in \cite{G-N} \S 2.2 
to Jacobi forms $\phi_{0,t}$ of weight $0$ and index $t$. 
For $t=2, 3, 4$, 
the $\rho_{K_{t}}$-valued modular form $f_{t}$ corresponding to $\phi_{0,t}$ 
has Fourier expansion 
$f_{t}(\tau) = q^{-1/4t}\mathbf{e}_{1}+a_{t}\mathbf{e}_{0}+ \cdots $ 
where 
$a_{t}=4, 2, 1$ for $t=2, 3, 4$ respectively. 
Thus $f_{t}$ for $t=2, 3, 4$ satisfy the condition in Proposition \ref{prop: LUE} (2). 
\end{example}

\subsection{Rational and integral part}\label{ssec: integral part}

We assume $\sigma(L)<0$ for simplicity. 
For a subring $R$ of ${\C}$ (typically ${\Z}$ or ${\Q}$) 
we write ${\MLR} \subset {\ML}$ 
for the group of modular forms $f$ whose principal part has coefficients in $R$. 
It is clear that 
${\MLZ}\otimes_{{\Z}}{\Q} = {\MLQ}$. 
Moreover, 
McGraw's rationality theorem \cite{Mc} and the Borcherds duality theorem 
imply that 
${\MLQ}\otimes_{{\Q}}{\C} = {\ML}$. 

\begin{proposition}\label{prop: rational part}
Let $\sigma(L)<0$. 
Then ${\MLQ}$ is closed under $\ast$. 
Hence it forms an associative ${\Q}$-algebra. 
\end{proposition}

This follows from the explicit calculation \eqref{eqn: explicit coefficient} of 
the Fourier coefficients of $f_{1}\ast f_{2}$, 
the fact that the $\Theta$-series $\Theta_{K^+}(\tau)$ has integral Fourier coefficients, 
and the following well-known fact. 

\begin{lemma}
Let $A$ be a finite quadratic module and $k<0$. 
If $f\in M^!_k(\rho_A)$ has rational principal part, 
every Fourier coefficient of $f$ is also rational. 
\end{lemma}

\begin{proof}
We supplement a proof for the convenience of the reader. 
Let $\Delta(\tau)$ be the $\Delta$-function. 
When $d\gg 0$, the product $\Delta^{d}f$ vanishes at the cusp, namely $\Delta^{d}f \in S_{k+12d}(\rho_A)$. 
By McGraw \cite{Mc}, the ${\C}$-linear space $S_{k+12d}(\rho_A)$ has a basis $f_1, \cdots, f_N$ with integral Fourier coefficients. 
We write $\Delta^{d} f=\sum_{i}a_if_i$ with $a_i\in {\C}$. 
Since $\Delta(\tau)$ has integral coefficients, 
the coefficients of $\Delta^{d}f$ in degree $<d$ are rational by the assumption on $f$. 
Since $k<0$, cusp forms in $S_{k+12d}(\rho_A)$ are determined by the coefficients in degree $<d$, 
so this implies that $a_i\in {\Q}$ for every $i$ by a standard argument. 
Since $\Delta^{-1}(\tau)$ has integral coefficients too, 
we find that $f=\sum_{i}a_i(\Delta^{-d}f_i)$ has rational coefficients. 
\end{proof}

\begin{remark}
Proposition \ref{prop: rational part} also holds in the case $\sigma(L)=0$ 
if we include the constant term $\sum_{\lambda}c_{\lambda}(0)\mathbf{e}_{\lambda}$ into the principal part. 
\end{remark}

Let $M^!(L)_{{\Z}}' \subset {\MLZ}$ be the group of modular forms $f$ 
whose all Fourier coefficients are integer. 
By the same reason as for Proposition \ref{prop: rational part}, 
we have 

\begin{proposition}
$M^!(L)_{{\Z}}'$ is closed under $\ast$ and hence forms a subring. 
\end{proposition}

In some cases, $M^!(L)_{{\Z}}' = {\MLZ}$ holds. 
This is the case when  
\begin{itemize}
\item $L=pU\oplus mE_8$ with $m>0$ (unimodular) 
\item $L=pU\oplus mE_8 \oplus \langle -2 \rangle$ 
\end{itemize}
as can be seen from the constructions of nice basis in \cite{D-J1}, \cite{D-J2} respectively. 
In both cases, we have in fact $c_0(0)\in 2{\Z}$ for $f\in {\MLZ}$ by \cite{W-W}. 
Therefore, when $p=2$, ${\MLZ}$ is identified with the multiplicative group of Borcherds products of integral weight. 
This means that we have new "unusual" product $\Psi(f_1\ast f_2)$ of 
two Borcherds products $\Psi(f_1), \Psi(f_2)$ of integral weight as a third one. 
When $L$ is not unimodular, ${\MLZ}$ also forms a nontrivial Lie algebra over ${\Z}$ 
under the Lie bracket \eqref{eqn: Lie bracket}. 
What is geometric interpretation of these new products?


\section{Finite generation}\label{sec: f.g.} 

In this section we prove that ${\ML}$ is finitely generated and 
give estimates, from above and below, 
on the number of generators. 
This section may be read independently of \S \ref{sec: first property}. 

\subsection{Finite generation}\label{ssec: f.g.}

In this subsection we prove 

\begin{proposition}\label{prop: f.g.}
The algebra ${\ML}$ is finitely generated over ${\C}$. 
\end{proposition}

For the proof we need the following construction. 

\begin{lemma}\label{lem: leading term}
There exists a natural number $d_{0}$ such that 
for any pair $(\lambda, n)$ with 
$\lambda\in A_{L}$ and $n\in q(\lambda)+{\Z}$, $n<-d_{0}$,  
there exists a modular form $f_{\lambda,n}\in {\ML}$ with Fourier expansion 
$f_{\lambda,n}(\tau)=q^{n}({\elambda}+\mathbf{e}_{-\lambda})+o(q^{n})$. 
\end{lemma}

\begin{proof}
For simplicity we assume $\sigma(L)<0$; 
the case $\sigma(L)=0$ can be dealt with similarly. 
For each natural number $d$ we let $V_{d}$ be 
the space of ${\C}A_{L}$-valued polynomials of the form 
\begin{equation}\label{eqn: principal part}
\sum_{\lambda\in A_{L}} 
\sum_{\substack{-d\leq m <0 \\ m\in q(\lambda)+{\Z}}} 
c_{\lambda}(m)q^{m}{\elambda}, \qquad 
c_{\lambda}(m) = c_{-\lambda}(m).  
\end{equation}
Then $\dim V_{d} = |A_{L}/\pm 1|\cdot d$. 
The filter $M^{!}(L)_{d}$ of ${\ML}$ is canonically embedded in $V_{d}$ 
by associating the principal parts. 
Let $S=S_{2-\sigma(L)/2}(\rho_{L}^{\vee})$ 
be the space of cusp forms of weight $2-\sigma(L)/2$ and type $\rho_{L}^{\vee}$. 
By the Borcherds duality theorem, 
we have the exact sequence 
\begin{equation*}
0 \to M^{!}(L)_{d}  \to V_{d}\to S^{\vee}. 
\end{equation*}
When $d\gg 0$, 
$V_{d}\to S^{\vee}$ is surjective (\cite{Bo00a}), 
and hence 
\begin{equation*}
\dim M^{!}(L)_{d} = |A_L/\pm 1|\cdot d - \dim S. 
\end{equation*}
In particular, we find that 
\begin{equation*}
\dim M^{!}(L)_{d+1} - \dim M^{!}(L)_{d} = |A_{L}/\pm 1|. 
\end{equation*}
On the other hand, 
$M^{!}(L)_{d}$ as a subspace of $M^{!}(L)_{d+1}$ 
is the kernel of the map 
$\rho_{d}\colon M^{!}(L)_{d+1} \to {\C}(A_{L}/\pm 1)$ 
that associates coefficients of the principal part in degree $\in [-d-1, -d)$. 
Therefore $\rho_{d}$ must be surjective when $d\gg 0$. 
The form $f_{\lambda,n}$ as desired can be obtained as 
$\rho_{d}^{-1}({\elambda}+\mathbf{e}_{-\lambda})$ for suitable $d$. 
\end{proof}

By the proof, $d_{0}$ can be taken to be the minimal degree $d$ where $V_{d}\to S^{\vee}$ is surjective. 
We now prove Proposition \ref{prop: f.g.}. 

\begin{proof}[(Proof of Proposition \ref{prop: f.g.})]
We first define a set of generators. 
First we take $f_{0}\in{\ML}$ whose Fourier expansion is of the form 
$q^{-d_{1}}\mathbf{e}_{0}+o(q^{-d_{1}})$ for some natural number $d_{1}$. 
Next, 
letting $d_{0}$ be as in Lemma \ref{lem: leading term}, 
we put  
\begin{equation*}
\Lambda_{1} = 
\{ \: f_{\lambda,m} \: | \: \lambda\in A_{L}/\pm 1, \: m\in q(\lambda)+{\Z}, \: -d_{0}-d_{1} \leq m < -d_{0} \: \}. 
\end{equation*}
Then we take a basis of $M^{!}(L)_{d_{0}}$ and denote it by $\Lambda_{2}$. 
We shall show that $f_{0}$, $\Lambda_{1}$ and $\Lambda_{2}$ 
generate ${\ML}$ as a ${\C}$-algebra. 

By definition $M^{!}(L)_{d_{0}+d_{1}}$ is generated by 
$\Lambda_{1}\cup \Lambda_{2}$ as a ${\C}$-linear space. 
The quotient 
$M^!(L)/M^!(L)_{d_0+d_1}$ 
is generated as a ${\C}$-linear space 
by any set of modular forms whose Fourier expansion is of the form 
$q^{n}({\elambda}+\mathbf{e}_{-\lambda})+o(q^{n})$ 
where $\lambda$ varies over $A_{L}/\pm 1$ and 
$n$ varies over  $q(\lambda)+{\Z}$ with $n<-d_{0}-d_{1}$. 
Therefore it suffices to show that 
we can construct such a modular form as a product of $f_{0}$ and elements of $\Lambda_{1}$. 
Since 
$f_{0}(\tau){\pushLK}=q^{-d_1}\mathbf{e}_{0}+o(q^{-d_1})$ 
and 
${\ThetaK}(\tau)=\mathbf{e}_{0}^{\vee}+o(1)$, 
we have 
$\xi(f_0)=q^{-d_1}+o(q^{-d_1})$. 
We take 
$m\equiv n$ modulo $d_1$ from 
$-d_0-d_1 \leq m < -d_0$ 
and put 
$r=(m-n)/d_1\in {\N}$. 
Then 
\begin{eqnarray*}
& & 
f_{0} \ast  \cdots \ast  f_{0} \ast  f_{\lambda,m} \qquad (f_{0} \: \: r \: \textrm{times}) \\ 
& = & 
(q^{-d_1}+o(q^{-d_1}))^{r} (q^{m}({\elambda}+\mathbf{e}_{-\lambda})+o(q^{m})) \\ 
& = & 
q^{n}({\elambda}+\mathbf{e}_{-\lambda}) + o(q^{n}). 
\end{eqnarray*}
This gives a desired modular form. 
\end{proof}

\begin{remark}
The rational part ${\MLQ}$ is also finitely generated as a ${\Q}$-algebra. 
The same proof works if we use the ${\Q}$-structure of $S$ given by McGraw's theorem \cite{Mc}. 
\end{remark}

\subsection{Bounds on the number of generators}

In this subsection we study upper and lower bounds on the minimal number of generators of ${\ML}$. 
We first determine the structure of ${\ML}$ as a $M^{!}$-module. 

\begin{proposition}\label{prop: free as module}
${\ML}\simeq (M^{!})^{\oplus (A_{L}/\pm1)}$ 
as a $M^{!}$-module. 
\end{proposition}

\begin{proof}
Let $W_{d}=M^!(L)_{d+1}/M^!(L)_{d}$. 
Taking coefficients of the principal part in degree $\in [-d-1, -d)$  
defines an embedding 
$\rho_{d}\colon W_{d}\hookrightarrow {\C}(A_{L}/\pm 1)$. 
On the other hand, multiplication by the $j$-function $j(\tau)=q^{-1}+O(1)\in M^{!}$ 
defines an injective map $W_{d}\hookrightarrow W_{d+1}$ 
which is compatible with $\rho_{d}$ and $\rho_{d+1}$. 
We thus have the filtration $(W_{d})_{d}$ of the space ${\C}(A_{L}/\pm 1)$ 
which stabilizes to ${\C}(A_{L}/\pm 1)$ in $d \gg 0$. 
Let 
\begin{equation*}
W_{0}=W_{d_{1}} \subsetneq W_{d_{2}} \subsetneq \cdots \subsetneq W_{d_{N}} = {\C}(A_{L}/\pm 1) 
\end{equation*}
be the reduced form of this filtration, 
namely $W_{d}=W_{d_{i}}$ in $d_{i}\leq d < d_{i+1}$. 

We take modular forms 
$\{ f_{ij} \}_{i,j}$, $1 \leq i \leq N$, 
such that 
$f_{ij}\in M^!(L)_{d_{i}}$ and 
$\{ \rho_{d_{i}}(f_{ij}) \}_{j}$ form a basis of $W_{d_{i}}/W_{d_{i-1}}$. 
Then 
$\{ \rho_{d_{i}}(f_{ij}) \}_{i,j}$ form a basis of ${\C}(A_{L}/\pm 1)$. 
We show that 
$\{ f_{\alpha} \}_{\alpha}$ 
freely generate ${\ML}$ as $M^!$-module. 
Since we have 
\begin{equation*}
M^!(L)_{d+1} = 
\langle M^!(L)_{d}, W_{d} \rangle = 
\langle M^!(L)_{d}, j\cdot M^!(L)_{d}, W_{d}/W_{d-1} \rangle, 
\end{equation*}   
induction on $d$ tells us that 
$M^!(L)_{d}\subset \sum_{\alpha} M^!\cdot f_{\alpha}$ 
for every $d$. 
Thus $\{ f_{\alpha} \}_{\alpha}$ generate ${\ML}$. 
If there was a relation 
\begin{equation}\label{eqn: freeness}
\sum_{\alpha} P_{\alpha}(j) f_{\alpha} = 0, \qquad P_{\alpha}\in {\C}[x], 
\end{equation}
then we would obtain a nontrivial ${\C}$-linear relation between 
$\{ \rho_{\alpha}(f_{\alpha}) \}_{\alpha}$ in ${\C}(A_L/\pm 1)$ 
by looking at the coefficients of the principal part of \eqref{eqn: freeness} in highest degrees. 
Thus $\{ f_{\alpha} \}$ are free generators. 
\end{proof}

This gives a lower bound of the number of generators of ${\ML}$ as algebra, 
which implies the following. 

\begin{proposition}\label{prop: finiteness}
Let $p\leq q$ be fixed. 
Let $N$ be a fixed natural number. 
Then up to isometry 
there are only finitely many pairs $(L, I)$ 
of an even lattice $L$ of signature $(p, q)$ and Witt index $p$ 
and a maximal isotropic sublattice $I\subset L$ such that 
the algebra $M^!(L, I)$ can be generated by at most $N$ elements. 
\end{proposition}

\begin{proof}
In \S \ref{ssec: theta product}, we observed that 
the dependence on $I$ is finite for a fixed lattice $L$. 
Hence it is sufficient to prove finiteness of lattices $L$. 
Since $f\ast g=\xi(f)\cdot g$, 
generators of ${\ML}$ as algebra 
also serve as generators as $M^!$-module. 
By Proposition \ref{prop: free as module}, 
we obtain the bound 
\begin{equation*}
N \: \geq \: |A_{L}/\pm 1| \: > \: |A_{L}|/2. 
\end{equation*}
Then our assertion follows from 
finiteness of even lattices of fixed signature and bounded discriminant. 
\end{proof}

Proposition \ref{prop: finiteness} is a consequence of the structure of ${\ML}$ as a $M^{!}$-module. 
It would be a natural problem whether the finiteness still holds even if we let $q$ vary with $p$ fixed. 
By Proposition \ref{prop: free as module}, 
the same statement is not true for generators \textit{as $M^!$-module}. 
(Take direct sum with the unimodular lattices $mE_{8}$.) 
So this could be one of touchstones for the theory of algebra structure on ${\ML}$.

Next we study upper bound of the number of generators of ${\ML}$ as algebra. 
By the proof of Proposition \ref{prop: f.g.}, 
we have the upper bound  
\begin{equation}\label{eqn: bound generator}
1 + d_{1}\cdot |A_{L}/\pm 1| + \dim M^!(L)_{d_{0}} 
\: \leq \: 
1+(d_0+d_1)\cdot |A_{L}/\pm 1|,   
\end{equation}
where $d_0$ and $d_1$ are as defined there. 
Clearly $d_1 \leq d_0+1$. 
We have the following upper bound of $d_0$. 
Let $V_{d}$ and $S$ be as in the proof of Lemma \ref{lem: leading term}. 
Recall that $d_0$ does not exceed the minimal degree where $V_d\to S^{\vee}$ is surjective. 

\begin{proposition}
The dimension $\dim {\rm Im}(V_d\to S^{\vee})$ 
is strictly increasing with respect to $d$ 
until $V_d\to S^{\vee}$ gets surjective. 
In particular, we have 
\begin{equation}\label{eqn: d_0 bound}
d_{0} \: \leq \: \dim S - d_{2}(|A_{L}/\pm 1|-1), 
\end{equation}
where 
$d_{2}=-[\sigma(L)/24+1]$ 
is the largest integer with $d_{2}<|\sigma(L)|/24$. 
\end{proposition} 

\begin{proof}
For the first assertion, what has to be shown is that 
$V_d\to S^{\vee}$ is surjective whenever 
${\rm Im}(V_d\to S^{\vee}) = {\rm Im}(V_{d+1}\to S^{\vee})$. 
By the Borcherds duality theorem, 
this condition means that 
the codimension of $M^!(L)_d$ in $V_d$ equals to 
the codimension of $M^!(L)_{d+1}$ in $V_{d+1}$. 
If we write $W_d=M^!(L)_{d+1}/M^!(L)_{d}$, 
we find that 
\begin{equation*}
\dim W_d = \dim (V_{d+1}/V_{d}) = |A_{L}/\pm 1|. 
\end{equation*}
As in the proof of Proposition \ref{prop: free as module}, 
this implies 
that 
$\dim W_{d'} = |A_{L}/\pm 1|$ 
for every $d'\geq d$. 

On the other hand, 
if $V_{d}\to S^{\vee}$ was not surjective, 
there must exist $d'\geq d$ such that 
${\rm Im}(V_{d'}\to S^{\vee}) \ne {\rm Im}(V_{d'+1}\to S^{\vee})$. 
By the same argument as above, 
this implies that 
\begin{equation*}
\dim W_{d'} < \dim (V_{d'+1}/V_{d'}) = |A_{L}/\pm 1|, 
\end{equation*}
which is absurd. 

The second assertion follows from Lemma \ref{lem: sgn bound} 
which implies the injectivity of $V_{d_{2}}\to S^{\vee}$. 
\end{proof}

By \eqref{eqn: bound generator} and \eqref{eqn: d_0 bound}, we obtain an upper bound 
for the minimal number of generators in terms of $\dim S$. 
An estimate of $\dim S$ is given in \cite{B-E-F}. 

We also note that 
$d_1 \geq |\sigma(L)|/24$ 
by the following well-known property. 

\begin{lemma}\label{lem: sgn bound}
If $M^!(L)_{d}\ne \{ 0 \}$, then $|\sigma(L)|\leq 24d$. 
\end{lemma}

\begin{proof}
If $f\ne 0 \in M^!(L)_{d}$, 
the product $\Delta^{d} f$ with the $\Delta$-function is holomorphic at the cusp, 
so its weight $\sigma(L)/2+12d$ must be nonnegative. 
\end{proof}

We close this subsection with some simple examples. 

\begin{example}
Assume that the obstruction space 
$S_{2-\sigma(L)/2}(\rho_{L}^{\vee})$ 
is trivial. 
(Such lattices $L$ with $p=2$ are classified in \cite{B-E-F}.) 
Then every polynomial as in \eqref{eqn: principal part} 
is the principal part of some modular form in ${\ML}$. 
In this case, using the notation in the proof of Proposition \ref{prop: f.g.}, 
we have 
$d_{0}=0$, 
$d_{1}=1$, 
$\Lambda_{2}=\emptyset$, 
and the modular form $f_{0}$ can be included in $\Lambda_{1}$. 
Therefore ${\ML}$ can be generated by modular forms 
$f_{\lambda}=q^n(\mathbf{e}_{\lambda}+\mathbf{e}_{-\lambda})+O(1)$ 
with $\lambda\in A_L/\pm1$ and $n\in q(\lambda)+{\Z}$, $-1\leq n < 0$. 
The minimal number of generators is thus equal to $|A_{L}/\pm1|$. 
The generator $f_{\lambda}$ with $\lambda\ne 0$ 
is either a left unit element or a left zero divisor 
according to Proposition \ref{prop: LUE} (2). 
\end{example}

\begin{example}\label{ex: generator}
We go back to Example \ref{ex: reflective 1} where $L=pU \oplus \langle -2 \rangle$. 
The algebra ${\ML}$ is generated by the two elements 
$f_{0}=q^{-1}\mathbf{e}_{0}+O(1)$ and 
$f_{1}=q^{-1/4}\mathbf{e}_{1}+O(1)$ 
with the relation 
$f_{1}\ast  f_{1}= 12 f_{1}$ and $f_{1}\ast  f_{0}= 12 f_{0}$.  
Thus the two basic reflective modular forms for $L$ 
give minimal generators of the algebra ${\ML}$.  
\end{example}


\section{Functoriality}\label{sec: functorial}  

In this section we prove that  
$\Theta$-product is functorial with respect to embedding of lattices 
if we use quasi-pullback as morphism. 
The statement is Proposition \ref{prop: functorial main}, 
and the proof is given in \S \ref{ssec: finite pullback} and  \S \ref{ssec: general case}. 
In \S \ref{ssec: functorial push} we also prove functoriality 
with respect to special pushforward. 
Except for Corollary \ref{cor: functorial add consequence}, 
this section may be read independently of \S \ref{sec: first property} and \S \ref{sec: f.g.}.

\subsection{Quasi-pullback}\label{ssec: quasi-pullback} 

Let $L$ be an even lattice of signature $(p, q)$ 
and $L'$ be a sublattice of $L$ of signature $(p, q')$. 
We do not assume that $L'$ is primitive in $L$. 
Following \cite{Ma}, we define a linear map 
$|_{L'} \colon {\ML}\to M^{!}(L')$ as follows. 
Let 
$N=(L')^{\perp}\cap L$, 
which is a negative-definite lattice. 
We write $N^{+}=N(-1)$. 
The lattice $L'\oplus N$ is of finite index in $L$. 
Let $f\in {\ML}$. 
We first take the pullback 
$f\!\uparrow_{L}^{L'\oplus N}$, 
which is an element of $M^!(L'\oplus N)$. 
Since 
$\rho_{L'\oplus N} = \rho_{L'}\otimes \rho_{N}$, 
we can take contraction of $f\!\uparrow_{L}^{L'\oplus N}$ 
with the $\rho_{N^{+}}$-valued theta series $\Theta_{N^{+}}$ of $N^{+}$. 
This produces a $\rho_{L'}$-valued weakly holomorphic modular form of weight $\sigma(L')/2$, 
which we denote by 
\begin{equation}\label{eqn: define quasi-pullback}
f|_{L'} = \langle f\!\uparrow_{L}^{L'\oplus N}, \: \Theta_{N^{+}} \rangle \quad \in M^{!}(L'). 
\end{equation}
We call $f|_{L'}$ the \textit{quasi-pullback} of $f$ to $L'$. 
The map $|_{L'} \colon {\ML}\to M^!(L')$ is $M^!$-linear. 

The geometric significance of this operation comes from Borcherds products as follows. 
Assume that $p=2$ and $f$ has integral principal part, 
and let $\Psi(f)$ be the Borcherds product associated to $f$ 
on the Hermitian symmetric domain $\mathcal{D}_{L}$ for $L$.   
The Hermitian symmetric domain $\mathcal{D}_{L'}$ for $L'$ 
is naturally embedded in $\mathcal{D}_{L}$. 
The quasi-pullback of $\Psi(f)$ from $L$ to $L'$, 
discovered by Borcherds \cite{Bo95}, \cite{B-K-P-SB},  
is defined by 
first dividing $\Psi(f)$ by suitable linear forms to get rid of zeros and poles containing $\mathcal{D}_{L'}$, 
and then restricting the resulting form to $\mathcal{D}_{L'}\subset \mathcal{D}_{L}$. 
It is proved in \cite{Ma} that this quasi-pullback of $\Psi(f)$ 
coincides with the Borcherds product for $f|_{L'}\in M^{!}(L')$ up to constant. 
Thus the operation $|_{L'}$ defined in \eqref{eqn: define quasi-pullback} 
can be thought of as a formal ${\C}$-linear extension of 
the quasi-pullback operation on Borcherds products. 

We can now state the main result of this \S \ref{sec: functorial}. 
We assume that $p\leq q' \leq q$ and both $L$ and $L'$ have Witt index $p$. 

\begin{proposition}\label{prop: functorial main}
Let $L'\subset L$ be as above. 
Let $I$ be a maximal isotropic sublattice of $L$ 
such that $I_{{\Q}}\subset L'_{{\Q}}$. 
We set $I'=I\cap L'$. 
Then we have 
\begin{equation*}\label{eqn: functorial}
(f|_{L'})\ast_{I'}(g|_{L'}) = |I/I'|\cdot (f\ast_{I}g)|_{L'} 
\end{equation*}
for $f, g \in {\ML}$. 
Therefore the map
\begin{equation*} 
|I/I'|^{-1}\cdot |_{L'} : M^!(L, I) \to M^{!}(L', I') 
\end{equation*}
is a ring homomorphism. 
\end{proposition}

This means that the assignment 
\begin{equation*}
(L, I) \mapsto M^{!}(L, I) 
\end{equation*}
is a contravariant functor 
from the category of pairs $(L, I)$ 
to the category of associative ${\C}$-algebras, 
by assigning the morphism $|I/I'|^{-1}\cdot |_{L'}$ 
to an embedding $(L', I')\hookrightarrow (L, I)$. 

The proof of Proposition \ref{prop: functorial main} is reduced to the following assertion. 

\begin{proposition}\label{prop: functorial}
Let $L'\subset L$ and $I'\subset I$ be as in Proposition \ref{prop: functorial main}. 
We put $K=I^{\perp}\cap L / I$ and $K'=(I')^{\perp}\cap L'/I'$. 
Let 
$\xi \colon {\ML}\to M^{!}$ and 
$\xi' \colon M^{!}(L')\to M^{!}$ be the maps 
$\xi= \langle \cdot {\pushLK}, {\ThetaK} \rangle$ and 
$\xi'= \langle \cdot \! \downarrow^{L'}_{K'}, \Theta_{(K')^{+}} \rangle$ 
respectively. 
Then we have 
\begin{equation*}
\xi' \circ |_{L'} = |I/I'| \cdot \xi. 
\end{equation*}
\end{proposition}

Indeed, if we admit Proposition \ref{prop: functorial}, we can calculate 
\begin{eqnarray*}
(f|_{L'})\ast_{I'}(g|_{L'}) 
& = & 
\xi'(f|_{L'})\cdot (g|_{L'}) 
= |I/I'| \cdot \xi(f) \cdot (g|_{L'}) \\ 
& = & 
|I/I'| \cdot (\xi(f)\cdot g)|_{L'} 
= |I/I'| \cdot (f\ast_{I}g)|_{L'}. 
\end{eqnarray*}
Thus Proposition \ref{prop: functorial} implies Proposition \ref{prop: functorial main}. 

Before proceeding, we note a consequence. 

\begin{corollary}\label{cor: functorial add consequence}
Let $\Theta^{\perp}(L)\subset {\ML}$ and 
$\Theta^{\perp}(L')\subset M^!(L')$ be 
the respective left annihilators. 
Then we have 
$|_{L'}^{-1}(\Theta^{\perp}(L'))=\Theta^{\perp}(L)$. 
In particular, we have 
${\rm Ker}(|_{L'})\subset \Theta^{\perp}(L)$. 
The map 
${\ML}/\Theta^{\perp}(L) \to M^!(L')/\Theta^{\perp}(L')$ 
induced by $|I/I'|^{-1}\cdot |_{L'}$ 
is inclusion of ideals in the polynomial ring $M^!={\C}[j]$. 
\end{corollary}

\begin{proof}
The equality 
$|_{L'}^{-1}(\Theta^{\perp}(L'))=\Theta^{\perp}(L)$ 
follows from Proposition \ref{prop: functorial}. 
Since $\xi$ and $\xi'$ embed 
${\ML}/\Theta^{\perp}(L)$ and $M^!(L')/\Theta^{\perp}(L')$ 
as ideals in $M^!$ respectively, 
the last assertion follows.  
\end{proof}

Thus the kernel of the quasi-pullback map $|_{L'}$ 
provides a natural example of two-sided ideal of ${\ML}$ contained in $\Theta^{\perp}$.

The proof of Proposition \ref{prop: functorial} occupies \S \ref{ssec: finite pullback} and \S \ref{ssec: general case}. 
It is divided into two parts, 
reflecting the fact that the quasi-pullback $|_{L'}$ is composition of 
two operators $\uparrow_{L}^{L'\oplus N}$ and $\langle \cdot, \Theta_{N^{+}} \rangle$. 
In \S \ref{ssec: finite pullback} we consider the case when $L'$ is of finite index in $L$. 
In \S \ref{ssec: general case} we consider the case when the splitting $L=L'\oplus N$ holds. 
The proof in the general case is a combination of these two special cases.

\subsection{The case of finite pullback}\label{ssec: finite pullback} 

In this subsection we prove Proposition \ref{prop: functorial} 
in the case when $L'$ is of finite index in $L$. 
In this case, the quasi-pullback $|_{L'}$ is the operation $\uparrow^{L'}_{L}$, 
and Proposition \ref{prop: functorial} takes the following form. 

\begin{lemma}\label{prop: finite pullback}
When $L'\subset L$ is of finite index, we have for $f \in {\ML}$ 
\begin{equation*}
\xi'(f\!\uparrow^{L'}_{L}) = |I/I'| \cdot \xi(f).  
\end{equation*}
\end{lemma}

This is a consequence of the following calculation in 
finite quadratic modules. 

\begin{lemma}\label{lem: finite pullback FQM}
Let $A$ be a finite quadratic module and 
$I_1, I_2 \subset A$ be two isotropic subgroups. 
We set $A_1=I_1^{\perp}/I_1$, $A_2=I_2^{\perp}/I_2$ and 
\begin{equation*}
A'= 
(I_{1}^{\perp}\cap I_{2}^{\perp}) / ((I_{1}\cap I_{2}^{\perp})+(I_{2}\cap I_{1}^{\perp})). 
\end{equation*}
Let $I_{2}'=I_{2}\cap I_{1}^{\perp}/I_{1}\cap I_{2}$ be the image of $I_{2}\cap I_{1}^{\perp}$ in $A_{1}$, 
and $I_{1}'=I_{1}\cap I_{2}^{\perp}/I_{1}\cap I_{2}$ be the image of $I_{1}\cap I_{2}^{\perp}$ in $A_{2}$. 
Then, under the natural isomorphism 
\begin{equation}\label{eqn: description of A'}
A' \simeq (I_{2}')^{\perp}\cap A_{1}/I_{2}' \simeq (I_{1}')^{\perp}\cap A_{2}/I_{1}', 
\end{equation}
we have 
\begin{equation}\label{eqn: pull push commute FQM}
\downarrow^{A}_{A_{2}} \circ \uparrow^{A}_{A_{1}} = 
| I_{1} \cap I_{2} |  \uparrow^{A_{2}}_{A'} \circ \downarrow^{A_{1}}_{A'} 
\end{equation}
as linear maps 
${\C}A_{1} \to {\C}A_{2}$. 
\end{lemma}

We postpone the proof of Lemma \ref{lem: finite pullback FQM} for a moment, 
and first explain how Lemma \ref{prop: finite pullback} is deduced from Lemma \ref{lem: finite pullback FQM}. 

\begin{proof}[(Proof of Lemma \ref{prop: finite pullback})]
Let $K=I^{\perp}\cap L /I$ and $K'=(I')^{\perp}\cap L' /I'$. 
We have a canonical embedding $K'\hookrightarrow K$ of finite index. 
Since ${\ThetaK}=\Theta_{(K')^{+}}\!\downarrow^{(K')^{+}}_{K^{+}}$, 
we find that  
\begin{equation*}
\xi(f) 
= \langle f{\pushLK}, {\ThetaK} \rangle 
= \langle f{\pushLK}, \: \Theta_{(K')^{+}}\!\downarrow^{(K')^{+}}_{K^{+}} \rangle 
= \langle f{\pushLK}\uparrow_{K}^{K'}, \: \Theta_{(K')^{+}} \rangle. 
\end{equation*}
On the other hand, we have 
\begin{equation*}
\xi'(f \! \uparrow^{L'}_{L}) = 
\langle f {\pullL}\downarrow^{L'}_{K'}, \: \Theta_{(K')^{+}} \rangle. 
\end{equation*}
Thus it is sufficient to show that 
\begin{equation}\label{eqn: pull push commute lattice}
\downarrow^{L'}_{K'} \circ {\pullL} = |I/I'| \uparrow^{K'}_{K} \circ {\pushLK} 
\end{equation}
as linear maps 
${\C}A_{L}\to {\C}A_{K'}$. 

We apply Lemma \ref{lem: finite pullback FQM} as follows. 
Let $I^{\ast}=I_{{\Q}}\cap L^{\vee}$ and 
$(I')^{\ast} = I_{{\Q}}\cap (L')^{\vee}$. 
We set 
$A=A_{L'}$, 
$I_{1}=L/L'$ and 
$I_{2}=(I')^{\ast}/I'$. 
Then 
$A_{1}\simeq A_{L}$ and  
$A_{2}\simeq A_{K'}$. 
We have 
$I_{2}\cap I_{1}^{\perp}=I^{\ast}/I'$ and 
\begin{equation*}
I_{1} \cap I_{2} 
= ( L\cap \langle L', (I')^{\ast} \rangle )/L' 
= \langle L', I \rangle /L' 
= I/I'. 
\end{equation*}
This implies that 
$I_{2}'=I^{\ast}/I \subset A_{L}$ 
and $A'= A_{K}$. 
Thus we have 
\begin{equation*}
\uparrow^{A}_{A_{1}} = \uparrow^{L'}_{L}, \quad 
\downarrow^{A}_{A_{2}} = \downarrow^{L'}_{K'}, \quad 
\downarrow^{A_{1}}_{A'} = \downarrow^{L}_{K}, \quad 
\uparrow^{A_{2}}_{A'} = \uparrow^{K'}_{K},  
\end{equation*}
hence \eqref{eqn: pull push commute FQM} implies \eqref{eqn: pull push commute lattice}.  
\end{proof}

We now prove Lemma \ref{lem: finite pullback FQM}. 

\begin{proof}[(Proof of Lemma \ref{lem: finite pullback FQM})] 
We first justify the isomorphism \eqref{eqn: description of A'}, 
which also implies that $A'$ is nondegenerate. 
We write 
$\hat{I}_{1}'=I_{1}\cap I_{2}^{\perp}$ and  
$\hat{I}_{2}'=I_{2}\cap I_{1}^{\perp}$. 
We shall establish the following commutative diagram: 
\begin{equation*}\label{eqn: CD}
\xymatrix{
    & (\hat{I}_2')^{\perp}\cap I_{1}^{\perp} \ar[rd]^{p_1} &  \\ 
(\hat{I}_{1}')^{\perp}\cap I_{2}^{\perp} \ar[rd]_{p_2} & 
I_{1}^{\perp}\cap I_{2}^{\perp} \ar[r]^{p_1'} \ar@{^{(}-_>}[u] \ar[d]^{p_2'} \ar@{_{(}-_>}[l]
& (I_{2}')^{\perp}\cap A_{1} \ar[d]_{q_{2}}  \\ 
   & (I_{1}')^{\perp}\cap A_{2} \ar[r]_{q_1} & A' 
}
\end{equation*}
Here $p_i$ is the quotient map by $I_{i}$ and 
$p_{i}'$ is the restriction of $p_{i}$. 
Since we have 
$\hat{I}_{1}' = I_{1} \cap (I_{1}^{\perp}\cap I_{2}^{\perp})$ and 
\begin{equation*}
I_{1} / \hat{I}_{1}'  \simeq 
((\hat{I}_2')^{\perp}\cap I_{1}^{\perp}) / (I_{1}^{\perp}\cap I_{2}^{\perp}) 
\simeq (\hat{I}_{2}')^{\perp}/I_{2}^{\perp}, 
\end{equation*}
we see that $p_{1}'$ is surjective and is the quotient map by $\hat{I}_{1}'$. 
This induces the map  
$q_{2}\colon (I_2')^{\perp}\cap A_{1} \to A'$ 
as the quotient map by $I_2'$. 
Similarly, we find that 
$p_2'$ is the quotient map by $\hat{I}_{2}'$ and 
$q_{1}$ is induced as the quotient map by $I_{1}'$. 

We now prove \eqref{eqn: pull push commute FQM}. 
Let $\lambda\in A_{1}$. 
It suffices to show that 
\begin{equation}\label{eqn: pullback FQM}
{\elambda}\uparrow_{A_{1}}^{A}\downarrow^{A}_{A_{2}} = 
|I_{1}\cap I_{2}| \cdot {\elambda}\downarrow^{A_{1}}_{A'}\uparrow^{A_{2}}_{A'}. 
\end{equation}
When $\lambda\not\in (I_{2}')^{\perp}$, we have 
$\mathbf{e}_{\lambda}\! \downarrow^{A_{1}}_{A'}=0$. 
On the other hand, we have 
$(\tilde{\lambda}, \hat{I}_{2}')\not\equiv 0$ 
for every $\tilde{\lambda}\in I_{1}^{\perp}$ 
in the inverse image of $\lambda$. 
In particular, we have $(\tilde{\lambda}, I_2)\not\equiv 0$ and hence 
$\mathbf{e}_{\tilde{\lambda}}\downarrow^{A}_{A_{2}}=0$. 
This implies that  
${\elambda}\uparrow_{A_{1}}^{A}\downarrow^{A}_{A_{2}}=0$.

Next let $\lambda\in (I_{2}')^{\perp}$. 
By the above commutative diagram, 
we can choose $\tilde{\lambda}\in I_1^{\perp}\cap I_2^{\perp}$ such that 
$p_{1}'(\tilde{\lambda})=\lambda$. 
Then 
\begin{equation}\label{eqn: push then pull}
\mathbf{e}_{\lambda}\downarrow^{A_{1}}_{A'} \uparrow_{A'}^{A_{2}} \: = \: 
\mathbf{e}_{q_{2}(\lambda)}\uparrow^{A_{2}}_{A'} \: = \: 
\sum_{\mu'\in I_{1}'} \mathbf{e}_{p_{2}'(\tilde{\lambda})+\mu'}. 
\end{equation}
On the other hand, we have 
\begin{equation}\label{eqn: pull then push}
\mathbf{e}_{\lambda}\uparrow_{A_{1}}^{A}\downarrow^{A}_{A_{2}} 
\: = \:  
\sum_{\mu\in I_{1}} \mathbf{e}_{\tilde{\lambda}+\mu}\downarrow^{A}_{A_{2}}  
\: = \: 
\sum_{\mu\in \hat{I}_{1}'} \mathbf{e}_{p_{2}(\tilde{\lambda}+\mu)} 
\: = \: 
\sum_{\mu\in \hat{I}_{1}'} \mathbf{e}_{p_{2}'(\tilde{\lambda})+p_{2}'(\mu)}.  
\end{equation}
Here we used the equality 
$(\tilde{\lambda}+I_{1}) \cap I_{2}^{\perp} = \tilde{\lambda}+\hat{I}_{1}'$. 
Since the map 
$p_{2}'\colon \hat{I}_{1}'\to I_{1}'$ 
is the quotient map by $I_{1}\cap I_{2}$, 
its fibers consist of $|I_{1}\cap I_{2}|$ elements. 
Comparing \eqref{eqn: push then pull} and \eqref{eqn: pull then push},  
we obtain the desired equality \eqref{eqn: pullback FQM}. 
\end{proof}

\subsection{The split case}\label{ssec: general case} 

Next we prove Proposition \ref{prop: functorial} 
in the case when the splitting $L=L'\oplus N$ holds. 
In this case, 
$\uparrow^{L'\oplus N}_{L}$ is identity, 
$I'$ coincides with $I$, 
so Proposition \ref{prop: functorial} takes the following form. 

\begin{lemma}\label{prop: split case}
When the splitting $L=L'\oplus N$ holds, we have for $f \in {\ML}$ 
\begin{equation*}
\xi'(\langle f, \Theta_{N^{+}} \rangle) = \xi(f).  
\end{equation*}
\end{lemma}

\begin{proof}
Since $K=K'\oplus N$, we have 
${\ThetaK}=\Theta_{(K')^{+}} \otimes \Theta_{N^+}$ 
under the natural isomorphism 
$\rho_{K^+}\simeq \rho_{(K')^{+}}\otimes \rho_{N^+}$. 
Therefore 
\begin{eqnarray*}
\xi'(\langle f, \Theta_{N^{+}} \rangle) 
& = & 
\langle \langle f, \Theta_{N^{+}} \rangle \! \downarrow^{L'}_{K'}, \: \Theta_{(K')^{+}} \rangle 
= 
\langle \langle f{\pushLK}, \Theta_{N^{+}} \rangle, \Theta_{(K')^{+}} \rangle \\ 
& = & 
\langle f{\pushLK}, \Theta_{N^{+}}\otimes \Theta_{(K')^{+}} \rangle 
= \xi(f). 
\end{eqnarray*}
This proves the desired equality.  
\end{proof}

We can now prove Proposition \ref{prop: functorial} in the general case. 

\begin{proof}[(Proof of Proposition \ref{prop: functorial})]
Let $L'\oplus N\subset L$ and $I'\subset I$ be as in Proposition \ref{prop: functorial}. 
We write 
$L''=L'\oplus N$,  
$K''=(I')^{\perp}\cap L''/I'$ and 
$\xi''= \langle \cdot\!\downarrow^{L''}_{K''}, \Theta_{(K'')^{+}} \rangle$. 
By using Lemma \ref{prop: split case} for $L'\subset L''$ and 
Lemma \ref{prop: finite pullback} for $L''\subset L$, 
we see that  
\begin{equation*}
\xi'(f|_{L'}) 
= \xi'(\langle f\!\uparrow_{L}^{L''}, \Theta_{N^{+}} \rangle ) 
= \xi''(f\!\uparrow_{L}^{L''}) 
= |I/I'|\cdot \xi(f). 
\end{equation*}
This proves Proposition \ref{prop: functorial} in the general case. 
\end{proof}

\subsection{Special finite pushforward}\label{ssec: functorial push}

$\Theta$-product is also covariantly functorial with respect to 
pushforward to a special type of overlattices. 
Let $I\subset L$ be as before. 

\begin{proposition}\label{prop: functorial push}
Let $L'$ be a sublattice of $L$ of finite index. 
Assume that $L= \langle L', I \rangle$. 
We set $I'=I\cap L'$. 
Then we have 
\begin{equation*}
(f\! \downarrow^{L'}_{L})\ast_{I}(g\! \downarrow^{L'}_{L}) = 
(f\ast_{I'}g)\! \downarrow^{L'}_{L} 
\end{equation*}
for $f, g\in M^!(L')$. 
\end{proposition}

\begin{proof}
We use the notation in the proof of Lemma \ref{prop: finite pullback}. 
Since $I_1=L/L'$ coincides with $I_1\cap I_2=I/I'$, 
we have $I_1 \subset I_2$. 
Hence $I_{1}'=\{ 0 \}$ 
and so the canonical embedding $K'\hookrightarrow K$ is isomorphic. 
Moreover, since $I_{2}=(I')^{\ast}/I'$ coincides with $I_{2}\cap I_{1}^{\perp}$, 
we have $(I')^{\ast}=I^{\ast}$ and hence 
$\langle L, I^{\ast} \rangle = \langle L', (I')^{\ast} \rangle$. 
These equalities imply that 
$\downarrow^{L'}_{K'} \: = \: \downarrow^{L}_{K} \circ \downarrow^{L'}_{L}$. 
Therefore we have 
\begin{equation*}
\xi'(f) = 
\langle f\! \downarrow^{L'}_{K'}, \: \Theta_{(K')^{+}} \rangle = 
\langle f\! \downarrow^{L'}_{L} \downarrow^{L}_{K}, \: \Theta_{K^{+}} \rangle = 
\xi(f\! \downarrow^{L'}_{L}). 
\end{equation*}
As in the case of quasi-pullback, 
this implies 
\begin{eqnarray*}
(f\! \downarrow^{L'}_{L})\ast_{I}(g\! \downarrow^{L'}_{L}) 
& = & 
\xi(f\! \downarrow^{L'}_{L})\cdot (g\! \downarrow^{L'}_{L}) 
\: = \: \xi'(f)\cdot (g\! \downarrow^{L'}_{L}) \\ 
& = & 
(\xi'(f)\cdot g)\! \downarrow^{L'}_{L} 
\: = \: (f\ast_{I'}g)\! \downarrow^{L'}_{L}. 
\end{eqnarray*}
This proves Proposition \ref{prop: functorial push}. 
\end{proof}


\end{document}